\documentclass[sn-mathphys-num]{sn-jnl}

\usepackage[T1]{fontenc}
\usepackage{amsmath}
\usepackage{amssymb}
\usepackage{graphicx,booktabs}
\usepackage{times}
\usepackage{enumitem}

\usepackage{algorithm}
\usepackage{algpseudocode}

\usepackage[utf8]{inputenc}
\usepackage[english]{babel}
\usepackage{xcolor}
\usepackage{soul}
\usepackage{cleveref}
\crefformat{equation}{(#2#1#3)}

\def\BA{{\bf A}}
\def\BAT{{\bf A}^{\!\!\top}}
\def\RR{\mathbb{R}}
\def\S{\mathbb{S}}

\def\BM{{\bf M}}

\def\tn{\mathop{{\widetilde{n}}}}

\newcommand{\bo}[1]{\mathbf{#1}}

\DeclareMathOperator{\rank}{rank}
\DeclareMathOperator{\svec}{svec}
\DeclareMathOperator{\smat}{smat}
\DeclareMathOperator{\trace}{trace}
\DeclareMathOperator{\Diag}{Diag}

\jyear{2021}%

\theoremstyle{thmstyleone}%
\newtheorem{theorem}{Theorem}
\newtheorem{lemma}[theorem]{Lemma}%

\theoremstyle{thmstyletwo}%
\newtheorem{Example}{Example}%
\newtheorem{remark}{Remark}%

\theoremstyle{thmstylethree}%
\newtheorem{definition}{Definition}%
\newtheorem{Assumption}{Assumption}%

\begin{document}

\title[Numerical solution of SDP relaxations of unconstrained 
BQP ]{On the numerical solution of  Lasserre relaxations of unconstrained 
binary quadratic optimization problem}

\author[1]{\fnm{Soodeh} \sur{Habibi}}\email{s.habibi@liverpool.ac.uk}
\author*[2,3]{\fnm{Michal} \sur{Ko\v{c}vara}}\email{m.kocvara@bham.ac.uk}
\author[4]{\fnm{Michael} \sur{Stingl}}\email{stingl@math.fau.de}
\affil[1]{\orgdiv{Department of Electrical Engineering and Electronics}, \orgname{University of Liverpool}, \orgaddress{\street{Brownlow Hill}, \city{Liverpool}, \postcode{L69 3GJ}, \country{UK}}}
\affil[2]{\orgdiv{School of Mathematics}, \orgname{University of Birmingham}, \orgaddress{\street{Edgbaston}, \city{Birmingham}, \postcode{B15 2TT}, \country{UK}}}
\affil[3]{\orgdiv{Faculty of Civil Engineering}, \orgname{ Czech Technical University in Prague}, \orgaddress{\street{Thákurova 2077/7}, \city{Prague 6}, \postcode{166 29},  \country{Czech Republic}}}
\affil[4]{\orgdiv{Applied Mathematics}, \orgname{Friedrich-Alexander-Universität Erlangen-Nürnberg}, \orgaddress{\street{Cauerstraße 11}, \city{Erlangen}, \postcode{91058}, \country{Germany}}}

\abstract{
    The aim of this paper is to solve linear semidefinite programs arising from higher-order Lasserre relaxations of unconstrained binary quadratic optimization problems. For this we use an interior point method with a preconditioned conjugate gradient method solving the linear systems. The preconditioner utilizes the low-rank structure of the solution of the relaxations. In order to fully exploit this, we need to re-write the moment relaxations. To treat the arising linear equality constraints we use an $\ell_1$-penalty approach within the interior-point solver.
    The efficiency of this approach is demonstrated by numerical experiments with the MAXCUT and other randomly generated problems and a comparison with a state-of-the-art semidefinite solver and the ADMM method. We further propose a hybrid ADMM-interior-point method that proves to be efficient for certain problem classes. As a by-product, we observe that the second-order relaxation is often high enough to deliver a globally optimal solution of the original problem.
}


\maketitle

\section{Introduction}\label{sec:intro}
Unconstrained binary quadratic optimization problems (UBQP) represent a surprisingly wide class of important optimization problems; see, e.g., the comprehensive overview \cite{kochenberger2014unconstrained}. The famous MAXCUT problem is a typical representative of this class. It is thus not surprising that they attract a great deal of attention among algorithm and software developers. Some of the most efficient algorithms for finding a global solution of UBQP combine branch-and-bound techniques with semidefinite programming (SDP) relaxations to obtain good lower bounds. These are represented by software like Biq Mac \cite{rendl2007branch}, BiqCrunch \cite{krislock2017biqcrunch} and BiqBin \cite{gusmeroli2022biqbin}.

The SDP relaxations are typically based on Shor's relaxation \cite{ben-tal-nemirovski} which is equivalent to the first-order Lasserre relaxation \cite{lasserre}. While higher-order relaxations would deliver much tighter lower bounds (if not exact solutions), the dimensions of the arising SDP problems are considered prohibitively large already for medium-sized UBQPs. For this reason, many authors proposed various techniques to strengthen the first-order relaxations; see, e.g., \cite{gvozdenovic2008operator,ghaddar2016dynamic,campos2022partial}
and the references therein.

It is not our goal to compete with the software mentioned in the first paragraph but to offer ways for potentially increasing their efficiency. In particular, our goal is to show that (at least) the second-order Lasserre relaxations are solvable by a specialized SDP software. We will also demonstrate that the second-order relaxations are, indeed, superior to the first-order ones and, in many cases, already deliver global solutions of the UBQP. For these problem we thus also provide an alternative approach to the branch-and-bound method, an approach of known complexity.

In this paper, we consider UBQP
\begin{equation}\label{eq:BQP}
\min_{x\in\RR^{s}} x^\top Q x \quad
\mbox{subject to}
\quad x_i \in {\cal B} \,,\quad i=1,\ldots,s
\end{equation}
with a symmetric matrix $Q\in\RR^{{s}\times{s}}$, where ${\cal B}$ is either the set $\{0,1\}$ or the set $\{-1,1\}$. \emph{We do not assume any sparsity in $Q$}, it is a generally dense matrix. In order to find a global optimum, we use Lasserre hierarchy of SDP problems---relaxations---of growing dimension \cite{lasserre}. The SDP relaxations have the form
\begin{align}
\label{eq:SDP}
&\min_{y\in\RR^n} q^\top y\\
&\mbox{subject to}~\begin{aligned}[t]&M(y):=\sum_{i=1}^{n}y_iM_i-M_0\succeq 0\,.
\end{aligned}\nonumber
\end{align}
Here $M$ is a so-called moment matrix, a (generally) dense matrix of a very specific form. In particular, if the solution of \cref{eq:BQP} is unique and the order of the relaxation is big enough, then $\rank M(y^*) \in \{1, 2\}$, where $y^*$ is a solution of \cref{eq:SDP} (see \Cref{sec:relax} below).

These problems are known to be difficult to solve due to the quickly growing dimension of the problem with the order of the relaxation; see, e.g., \cite{kim-kojima}. Since we do not assume any sparsity in $Q$, we cannot use sparse techniques, such as those proposed in \cite{sparsePOP,victor1,victor2}. Instead, we propose to use our recently developed software \mbox{\emph{Loraine}}~\cite{Loraine}. Loraine uses a primal-dual predictor-corrector interior-point (IP) method together with an iterative algorithm for the solution of the resulting linear systems. The iterative solver is a preconditioned Krylov-type method with a preconditioner utilizing low rank of the solution.

Several authors observed and confirmed by numerical experiments that SDP reformulation of UBQP (the first-order relaxation) can be efficiently solved by an SDP variant of the ADMM method, see \cite{wen2010alternating}. We will show that not only this observation can be extended to higher-order relaxations but that inexact ADMM results can be efficiently used as a warm start for the IP algorithm in Loraine. This is due to the choice of the preconditioner used within the iterative solver in Loraine. Based on this observation, we will propose a new, \emph{hybrid ADMM-Loraine algorithm} that, for a certain class of problems, will be superior to the single algorithms.

The paper is organized as follows. 
Section~2 introduces the low-rank SDP solver Loraine and the assumptions needed for its efficiency. 
In Section~3 we discuss the various forms of SDP relaxations of UBQP. Then, in Section~4 we briefly describe the ADMM algorithm for SDP and introduce the new hybrid algorithm, using an inexact ADMM result as a warm start for Loraine. The last Section~5 is devoted to numerical experiments using instances of the MAXCUT problems and randomly generated UBQPs.

\subsection{Our contribution}

We present a reformulation of \cref{eq:SDP} into a structure suitable for our IP solver Loraine. This reformulation uses an $\ell_1$-penalty function and leads to a low-rank solution of the form needed by the preconditioned iterative solver in Loraine. As a result, we can solve much larger problems than standard SDP solvers applied to formulation \cref{eq:SDP}.

We propose a novel ADMM-warm-started interior-point method. For a warm start to be beneficial for IP, it has to provide an approximation of both, primal and dual solution. This is obtained by ADMM with a very low accuracy stopping criterion. Moreover, ADMM also gives a very good estimate of the rank of the solution, thus making the preconditioner in Loraine very efficient.

In our numerical experiments, we not only demonstrate the efficiency of the proposed approach, we also confirm the observation in \cite{kim-kojima} that second-order relaxations are, in many cases, sufficient to obtain an exact solution of UBQP.

\subsection{Notation}
We denote by $\S^m,\S^m_+$ and $\S^m_{++}$, respectively, the space of $m\times m$ symmetric matrices, positive semidefinite and positive definite matrices.
The notation ``$\svec$'' and ``$\smat$" refer to the symmetrized vectorization and its inverse operation, respectively. 
The symbol $\bullet$ denotes the Frobenius inner product of two matrices, $A \bullet B=\trace(A^{\!\top} \!B)$.
Finally, the notation $e_n$ (or just $e$) is used for the vector of all ones.

\section{The solver Loraine}
Loraine\footnote{github.com/kocvara/Loraine.m} is a general-purpose solver for any linear SDP developed by the authors and implemented in \mbox{MATLAB} and Julia. Compared to other general-purpose SDP software, it particularly targets at problems with low-rank solutions. To solve the arising systems of linear equations, it uses the preconditioned conjugate gradient method, as described in detail in~\cite{Loraine}. The preconditioner, introduced in \cite{Zhang_2017}, is based on the assumption that the solution matrix has a small number of outlying eigenvalues. The user can choose between the direct and iterative solver, the type of preconditioner and the expected rank of the solution.

Loraine was developed for problems of the type
\begin{align}
\label{eq:SDP-P}
&\min_{y\in\RR^n\!,\,S\in\S^m\!,\,s_{\textrm{lin}}\in\RR^\nu~} b^\top y\\
&\mbox{subject to}~\begin{aligned}[t]&\sum_{i=1}^{n}y_iA_i+S=C\\
&Dy+s_{\textrm{lin}}=d\\
&S\succeq 0,\ 
s_{\textrm{lin}}\geq 0
\end{aligned}\nonumber
\end{align}
with the Lagrangian dual
\begin{align}
\label{eq:SDP-D}
&\max_{X\in\S^m\!,\,x_{\textrm{lin}}\in\RR^\nu~} C \bullet X +d^\top x_{\textrm{lin}}\\
&\mbox{{subject to}}\ \begin{aligned}[t]&A_i \bullet X+ (D^\top x_{\textrm{lin}})_i=b_i, \quad i=1,\dots,n\\
&X\succeq 0,\  x_{\textrm{lin}}\geq 0\,.
\end{aligned}\nonumber
\end{align}
In the following, we call \cref{eq:SDP-P} a problem of \emph{primal form} and \cref{eq:SDP-D} a problem of \emph{dual form}.

Loraine is efficient under the following assumptions:
\begin{Assumption}\label{assum:a2}
	Problems \cref{eq:SDP-P},\cref{eq:SDP-D} are strictly feasible, i.e., there exist $X\in\S^m_{++}$, $x_{\textrm{lin}}\in\RR^\nu_{++}$, $y\in\RR^n$, $S\in\S^m_{++}$, $s_{\textrm{lin}}\in\RR^\nu_{++}$, such that $A_i \bullet X + (D^\top x_{\textrm{lin}})_i=b_i $, $\sum_{i=1}^{n}y_iA_i+S=C$ and $Dy+s_{\textrm{lin}}=d$ (Slater's condition).
\end{Assumption}
\begin{Assumption}\label{assum:a3}
	Define the matrix $\BA=[\svec A_1,\dots, \svec A_n]$. We assume that any matrix-vector products with $\BA$ and $\BAT $ may each be computed in $\mathcal{O}(n)$ flops and memory.
\end{Assumption}
\begin{Assumption}\label{assum:a4} The inverse $(D^\top D)^{-1}$ exists and $(D^\top D)^{-1}$ together with the matrix-vector product with $(D^\top D)^{-1}$ may each be computed in $\mathcal{O}(n)$ flops and memory.
\end{Assumption}
\begin{Assumption}\label{assum:a5} The dimension of $X$ is much smaller than the number of constraints in \cref{eq:SDP-D}, i.e., \mbox{$m\ll n$}.
\end{Assumption}
\begin{Assumption}\label{assum:alr}
	Let $X^*$ be the solution of \cref{eq:SDP-D}. We assume that $X^*$ has $k$ outlying eigenvalues, i.e., that
	$$
	(0\leq)\;\lambda_1(X^*)\leq\cdots\leq\lambda(X^*)_{m-k}\ll \lambda(X^*)_{m-k+1}\leq\cdots\leq\lambda(X^*)_{m}\,,
	$$
	where $k$ is very small, typically smaller than 10 and, often, equal to 1.
	This includes the particular case when the rank of $X^*$ is very small.
\end{Assumption}

The last Assumption~\ref{assum:alr} is not satisfied by problem \cref{eq:SDP} so, in the next section, we will:
\begin{itemize}
	\item[(i)] re-write SDP relaxation \cref{eq:SDP} in the dual form \cref{eq:SDP-D} by introducing auxiliary variables and additional linear equality constraints;
	\item[(ii)] treat the new linear equality constraints by $\ell_1$-penalty approach, in order to replace equalities by inequalities, as required by the interior-point algorithm;
	\item[(iii)] show that the matrix associated with the new linear inequality constraints is block diagonal with blocks of very small size and thus satisfies \Cref{assum:a3};
	\item[(iv)] show that \Cref{assum:a5} is naturally satisfied for the re-written problem.
\end{itemize}

\section{Forms of SDP relaxations}\label{sec:relax}
\subsection{Lasserre hierarchy of moment problems}
Let ${\cal B} = \{-1,1\}$. By introducing a matrix variable $X = x x^\top$,
problem \cref{eq:BQP} can be equivalently written as the following SDP problem with a rank constraint:
\begin{align}
\label{eq:BQPSDP}
&\min_{X\in\S^{s}} {Q}\bullet X\\
&\mbox{subject to}~
\begin{aligned}[t]&X_{i,i} = 1,\quad i=1,\ldots,s\\
& X\succeq 0\\
&\rank X = 1\,.
\end{aligned}\nonumber
\end{align}

\begin{remark}
	For ${\cal B} = \{0,1\}$, we first use the substitution $\tilde{x} = 2x-1$ to get a problem with objective function
	$\tilde{x}^\top Q\tilde{x}+2e^\top Q\tilde{x}$ and constraints $\tilde{x}_i\in\{-1,1\}$. Then we set 
	$\widetilde{Q} = \begin{pmatrix} 0 &e^\top Q  \\ Qe&Q\end{pmatrix}\!\in\S^{s+1}$, introduce a variable 
	$\widetilde{X}\in\S^{s+1}$, $\widetilde{X}= \begin{pmatrix}1\\ \tilde{x}\end{pmatrix}\begin{pmatrix}1\ \ \tilde{x}^\top \end{pmatrix}$ and solve problem \cref{eq:BQPSDP} with $Q,X$ replaced by $\widetilde{Q},\widetilde{X}$.
\end{remark}

In view of the above remark, \emph{in the rest of the paper we will only consider the case ${\cal B} = \{-1,1\}$.}

For $\omega\in\mathbb{N}$, $\omega\geq 1$, let
$
\mathbb{N}^s_\omega 
$
the set of multi-indices $\alpha\in\mathbb{N}^s$ with $\sum_{i=1}^s\alpha_i\leq \omega$ and $\alpha \not= \mathbf{0}_s$.

\begin{definition} Given an integer $\omega\geq 1$ and a sequence ${\bo y} =({\bo y}_\alpha)_{\alpha \in \mathbb{N}^s_{2\omega}}$, its \emph{{moment matrix of order $\omega$}} is the matrix ${\bo M}_\omega({\bo y})$
	indexed by $\mathbb{N}^s_{\omega}$ with $(\alpha,\beta)$th entry ${\bo y}_{\alpha+\beta}$ for $\alpha,\beta\in\mathbb{N}^n_\omega$, according to the graded inverse lexicographic order.
\end{definition}
Finally, we associate the unique (multi-indexed) elements ${\bo y}_\gamma$ of ${\bo M}_\omega({\bo y})$ with elements of a vector $y$. 

Hence, looking first at the first-order relaxation, we assign to each unique component of the matrix $X = x x^\top$ a variable $y_k$. Considering the constraints $X_{i,i} = 1$ and the symmetry of $X$, we thus have $y\in\RR^r$ with $r=s(s-1)/2$. Replacing the variable $X$ by $y$, the constraint $X\succeq 0$ can be written as 
$$
	{\bo M}_1(y):=\sum_{i=1}^r M_iy_i + I \succeq 0,
$$
where $M_i\in\S^{s}$ are suitable matrices; for details, see \cite{lasserre}.

\begin{Example} \label{ex:21}
	For instance, for $s=3$ we use the vector (monomial basis) $\begin{pmatrix}x_1,\; x_2,\; x_3\end{pmatrix}^{\!\top}$; then $r=3$ and
	\begin{align*}
	M_1 &= \begin{pmatrix}
	0&1&0\\1&0&0\\0&0&0
	\end{pmatrix},\ 
	M_2 = \begin{pmatrix}
	0&0&1\\0&0&0\\1&0&0
	\end{pmatrix},\ 
	M_3 = \begin{pmatrix}
	0&0&0\\0&0&1\\0&1&0
	\end{pmatrix}
	\end{align*}
	and the moment matrix ${\bo M}_1(y)$ is associated with the original variables as
	$$
	\left({\bo M}_1(y):=\right) \sum_{i=1}^r M_iy_i + I =\begin{pmatrix}
	1&y_1&y_2\\y_1&1&y_3\\y_2&y_3&1
	\end{pmatrix} \longleftrightarrow
	\begin{pmatrix}x_1\\x_2\\x_3\end{pmatrix}
	\begin{pmatrix}x_1\\x_2\\x_3\end{pmatrix}^{\!\!\!\top}=
	\begin{pmatrix}
	x_1^2&x_1x_2&x_1x_3\\x_1x_2&x_2^2&x_2x_3\\x_1x_3&x_2x_3&x_3^2
	\end{pmatrix}.
	$$
\end{Example}

By ignoring the rank constraint, we can now define the first-order relaxation (Shor's relaxation) of \cref{eq:BQPSDP} as the following \emph{moment problem}:
\begin{align}
\label{eq:BQPSDPrel1}
&\min_{y\in\RR^r} q^\top {y}\\
&\mbox{subject to}~
\begin{aligned}[t]&{\bo M}_1(y)\succeq 0\,,
\end{aligned}\nonumber
\end{align}
where ${q}=\svec{Q}$.

The Lasserre hierarchy consists in adding higher-order monomials in the monomial basis and, therefore, the higher-order relaxation problems \emph{will all be of the form}~\cref{eq:BQPSDPrel1} with ${\bo M}_1(y)$ replaced by ${\bo M}_\omega(y)$.

Now, using the constraint $x_i^2=1$, we do not need to add all higher-order monomials, only the mixed terms. This means that the size of ${\bo M}_\omega(y)$ will be equal to $\sum_{i=1}^\omega {s\choose i}$.
Furthermore, as some of the terms in the higher-order moment matrix are repeated and some can be simplified using the  constraint $x_i^2=1$ (e.g., $x_1x_2^2$ is reduced to $x_1$) and thus even more terms are repeated, the number of variables in the higher-order relaxations is smaller than the number of the elements of the upper triangle of the matrix (see also Example~\ref{ex:22} below).

\begin{Example}
	The second-order relaxation in \Cref{ex:21} will lead to a $(6\times 6)$ moment matrix associated with the dyadic product of the vector of monomials of order up to two $(x_1,x_2,x_3,x_1x_2,x_1x_3,x_2x_3)^{\!\top}$ with itself. As explained above, we excluded the monomials $x_1^2,x_2^2,x_3^2$ from the list. The total number of variables $y_i$ is 7, as shown in Example~\ref{ex:22} below.
\end{Example}

%
\begin{definition}
	The \emph{Lasserre relaxation of order $\omega$} of the problem \cref{eq:BQP} is given by problem~\cref{eq:BQPSDPrel1} with ${\bo M}_1(y)$ replaced by ${\bo M}_\omega(y)$ and
	with general dimensions $n,m$, replacing $r,s$, respectively; i.e., $y\in\RR^n$ and ${\bo M}_\omega(y)\in\S^m$.
\end{definition}
The dimensions $n,m$ grow quickly with the order of the relaxation. In particular, we have
\begin{equation}\label{eq:dim}
\frac{s^4}{24} \lessapprox n \leq 2^s - 1 \quad \mbox{and}\quad m = \sum_{i=1}^\omega {s\choose i}\,.
\end{equation}
(The authors are not aware of any exact formula for $n$ or the lower bound on $n$; the lower bound in \cref{eq:dim} is our estimate for $\omega=2$ based on numerical experiments. Obviously, the higher $\omega$ the bigger $n$ until it reaches the ``saturation point'' $2^s-1$.) 

The dimension of the problem makes it very challenging for standard SDP solvers based on second-order methods, such as interior-point methods; see \cite{kim-kojima} and the numerical examples in the last section of this paper.

It was shown by Fawzi et al.\ \cite{fawzi} and Sakaue et al.\ \cite{sakaue}, and numerically confirmed in \cite{kim-kojima}, that for this type of problems, the sequence of these approximations is \emph{finite} and the upper bound on the order of the relaxation to obtain {exact} solution of \cref{eq:BQP} is $\lceil s/2\rceil$. Laurent \cite{laurent} showed that this is also a lower bound for MAXCUT problems with unweighted complete graphs.

To determine whether a relaxation of order $\omega$ is exact, i.e., whether its solution is the solution of  \cref{eq:BQP}, we use the rank of the moment matrix; see \cite[Thm.\;6.19]{lasserre}. In particular, assuming that \cref{eq:BQP} has $p$ global minimizers, the rank of the moment matrix will be less than or equal to $p$. Therefore, if we assume that \cref{eq:BQP} has a unique solution, then the moment matrix of the exact relaxation will be of rank one. 

\begin{remark}
	Notice that, when ${\cal B}=\{-1,1\}$, every global minimizer $x$ will have a ``symmetric'' counterpart $-x$, as there are only quadratic terms in the problem. In this case, by unique minimizer, we understand ``unique up to the multiple by $-1$'' and the rank of the exact relaxation will be two.
\end{remark}

Above, we spoke about an ``exact relaxation'' and a ``rank of the moment matrix''. When solving the problem by an interior point method, we can only speak about a {numerical rank} defined below.
\begin{definition}
	Consider $A\in\S^m_+$ with eigenvalues $\lambda_i(A)$ and with the largest eigenvalue $\lambda_{\textrm{max}}(A)$. Let $\varepsilon>0$. The \emph{numerical rank} of $A$ is defined as
	$$
	r_\varepsilon(A) = \vert{\cal I}\vert,\quad {\cal I} = \left\{i\in\{1,\ldots,m\}\mid \frac{\lambda_i(A)}{\lambda_{\textrm{max}}(A)}\geq \varepsilon\right\}\,.
	$$
\end{definition}
\begin{definition}\label{def:ex}
	For problems with a unique solution and ${\cal B}=\{-1,1\}$, the relaxation of order $\omega$ is called \emph{exact} if the numerical rank of the corresponding moment matrix is less than or equal to two.
\end{definition}



\subsection{Re-writing the problem}
As mentioned in the Introduction, problem \cref{eq:BQPSDPrel1} is in the ``wrong'' form for our solver Loraine. While the solution $M^*:=M(y^*)$ of \cref{eq:BQPSDPrel1} has low rank (in particular, $\rank M^* = 2$ when the solution of \cref{eq:BQP} is unique), Loraine expects the solution of an SDP problem in the \emph{dual form} \cref{eq:SDP-D} to have low rank. Our next goal is thus to rewrite \cref{eq:BQPSDPrel1} in the dual form; in other words, to rewrite the dual to \cref{eq:BQPSDPrel1} in the \emph{primal form} \cref{eq:SDP-P}.

The dual to \cref{eq:BQPSDPrel1} is the problem
\begin{align*}
&\max_{Z\in \S^{m}} -I\bullet Z\\
&\mbox{subject to}~
\begin{aligned}[t]&M_i\bullet Z = {q}_i,\quad i=1,\ldots,n\\
&Z\succeq 0\,,
\end{aligned}\nonumber
\end{align*}
and, by vectorizing the matrices as $z=\svec Z$ and $\BM= (\svec M_1, \ldots,\svec M_n)^\top$, $\BM\in\RR^{n\times\tn}$, it can be further written as 
\begin{align}
\label{eq:BQPSDP2}
&\min_{z\in \RR^{\tn}} (\svec I)^\top z\\
&\mbox{subject to}~
\begin{aligned}[t]&\smat(z) \succeq 0\\
&\BM z = {q}
\end{aligned}\nonumber
\end{align}
with $\tn = m(m+1)/2$.
Now, from the fact that the problem \eqref{eq:BQPSDP2} is the dual problem to \eqref{eq:BQPSDPrel1}, we have the following lemma:
\begin{lemma}
	\label{lem:y1}
	The vector $y$ in \eqref{eq:BQPSDPrel1} is the Lagrangian  multiplier to $\BM z = {q}$ in \eqref{eq:BQPSDP2}.
\end{lemma}
Problem \cref{eq:BQPSDP2} is now almost in the right form for our solver: it is in the primal form and the dual solution is a low-rank matrix. Notice that, despite $Q$ being possibly a dense matrix, the data in \cref{eq:BQPSDP2} are always very sparse. One obstacle remains---the linear equality constraints in \cref{eq:BQPSDP2}. We are addressing it in the next section.


\subsection{Equality constraints by $\ell_1$-penalty}
Interior-point methods have been designed for optimization problems with inequality constraints and cannot directly handle equality constraints. There are various ways how to overcome this, starting with writing equality as two inequalities; see \cite{anjos} for an overview focused on SDP interior-point algorithms. In our algorithm, we treat the equality constraints by the following $\ell_1$-penalty approach.

Introducing $\ell_1$-penalty for the linear equality constraints in \cref{eq:BQPSDP2}, we obtain the following problem
\begin{align}
\label{eq:BQPSDP2pen}
&\min_{z\in \RR^{\tn}} (\svec I)^\top z + \mu \|\BM z - {q}\|_1\\
&\mbox{subject to}~
\begin{aligned}[t]&\smat (z) \succeq 0\,,
\end{aligned}\nonumber
\end{align}
with a penalty parameter $\mu>0$. Recall that this penalty is exact, in the sense that there exists $\mu_0>0$ such that the solution of \cref{eq:BQPSDP2pen} with $\mu\geq \mu_0$ is equivalent to the solution of \cref{eq:BQPSDP2} (see, e.g., \cite{di1989exact} or \cite[Chap.5.3]{geiger-kanzow}).

We now introduce two new variables, $r\in\RR^n$, $s\in\RR^n$, satisfying
$$
\BM z - {q} = r-s,\quad r\geq0,\ s\geq 0 \,.
$$
After substitution, \cref{eq:BQPSDP2pen} reads as 
\begin{align}
\label{eq:BQPSDP2pen1}
&\min_{z\in \RR^{\tn},\,r\in\RR^n,\,s\in\RR^n} (\svec I)^\top z + \mu \sum_{i=1}^n(r_i+s_i)\\
&\mbox{subject to}~
\begin{aligned}[t]&\smat(z) \succeq 0\\
&\BM z - {q} = r-s\\
&r\geq0,\ s\geq 0\,.
\end{aligned}\nonumber
\end{align}
Using the identity $r=\BM z - {q}+s$ to eliminate variable $r$, we arrive at our final problem
\begin{align}
\label{eq:BQPSDP2pen2}
&\min_{z\in \RR^{\tn},\,s\in\RR^n} (\svec I)^\top z + \mu \sum_{i=1}^n\left((\BM z - {q})_i +2s_i\right)\\
&\mbox{subject to}~
\begin{aligned}[t]& \smat(z) \succeq 0\\
&\BM z - {q}+s \geq 0\\
&s\geq 0\,.
\end{aligned}\nonumber
\end{align}
Problem \cref{eq:BQPSDP2pen2} is now in the primal form \cref{eq:SDP-P}.

By Lemma \ref{lem:y1}, we know that $y$ in \eqref{eq:BQPSDPrel1} is the Lagrangian multiplier to the equality constraint in \eqref{eq:BQPSDP2}. The question is how to obtain $y$ from \eqref{eq:BQPSDP2pen2}, the $\ell_1$-penalty reformulation of \eqref{eq:BQPSDP2}.
\begin{lemma}\label{th:lemma}
	A solution $y$ of \eqref{eq:BQPSDPrel1} can be obtained from the solution of \eqref{eq:BQPSDP2pen2} as $y_i=\mu+\lambda_i,\ i=1,\dots,n$, in which $\mu$ is the penalty parameter at the solution and $\lambda$ is the Lagrangian multiplier to the constraint $\BM z - {q}+s \geq 0$.
\end{lemma}
\begin{proof} We start with the optimality conditions of the problem \eqref{eq:BQPSDP2pen2}. Using suitable matrices ${\cal E}_i\in\S^m$, we can write
	$\smat(z) = \sum_{i=1}^{\tn} {\cal E}_i z_i$.
	The Lagrangian of \eqref{eq:BQPSDP2pen2} is then
	\begin{align*}
	\mathcal{L}_1(z,s,\Gamma,\lambda,\eta)= &-\left(f(z)+\mu \sum_{i=1}^n\left((\BM z - {q})_i +2s_i\right)\right)- \Gamma \bullet \sum_{i=1}^{\tn} {\cal E}_i z_i\\
	& - \sum_{i=1}^n \lambda_i\left(\sum_{j=1}^{\tilde{n}}\BM_{ij}z_j-{q}_i+s_i\right)- \sum_{i=1}^n \eta_i s_i,
	\end{align*}
	where $\Gamma\succeq 0$ and $\lambda,\eta\geq 0$ are the corresponding Lagrangian multipliers and $f(z)=(\svec I)^\top z$. By using the KKT conditions at the solution we have
	\begin{align*}
	&\nabla_{\!z}\,\mathcal{L}_1(z^*,s^*,\Gamma^*,\lambda^*,\eta^*)=0:\\  
	&-(\nabla f(z^*))_j-\mu\sum_{i=1}^{n}\BM_{ij}
	-\sum_{i=1}^{n}\BM_{ij}\lambda^*_i
	-\Gamma^*\bullet {\cal E}_j = 0,\quad j=1,\ldots,\tn
	\end{align*}
	and so
	\begin{align}
	\label{eq:lagl1}
	(\nabla f(z^*))_j = 
	-\sum_{i=1}^{n}\BM_{ij}(\mu + \lambda^*_i)-\Gamma^*\bullet {\cal E}_j,\quad j=1,\ldots,\tn.
	\end{align}
	Now, the Lagrangian of the problem \eqref{eq:BQPSDP2} is
	\begin{align*}
	\mathcal{L}_2(z,\Gamma,y)= -f(z)-\Gamma \bullet \sum_{i=1}^{\tn} {\cal E}_i z_i - \sum_{i=1}^n y_i\left(\sum_{j=1}^{\tilde{n}}\BM_{ij}z_j-{q}_i\right),
	\end{align*}
	where, as above, $\Gamma\succeq 0$ and $y\geq 0$ are the corresponding Lagrangian multipliers. From the KKT conditions, we have
	\begin{align*}
	&\nabla_{\!z}\,\mathcal{L}_2(z^*,\gamma^*,y^*)=0:\\
	& -(\nabla f(z^*))_j-\Gamma^*\bullet {\cal E}_j-\sum_{i=1}^{n}\BM_{ij}y^*_i =0,\quad j=1,\ldots,\tn&
	\end{align*}
	so
	\begin{align}
	\label{eq:lag}
	(\nabla f(z^*))_j=-\sum_{i=1}^{n}\BM_{ij}y^*_i-\Gamma^*\bullet {\cal E}_j,\quad j=1,\ldots,\tn.
	\end{align}
	Thus, from \eqref{eq:lagl1} and \eqref{eq:lag}, $y$ can be written as
	\begin{align}
	\label{eq:y}
	y_i=\mu+\lambda_i,\quad i=1,\dots,n
	\end{align}
	and we are done.
\end{proof}

The question remains how to update the penalty parameter $\mu$ in order to get an exact solution of problem \cref{eq:BQPSDP2} using our interior-point solver. Here we will use that fact that the $\ell_1$ penalty is exact when $\mu$ is greater than the $\ell_\infty$-norm of the Lagrangian multiplier associated with the penalized constraint at a KKT point; see, e.g., \cite[Thm.5.11]{geiger-kanzow}. 

We have tested the following two options.
\begin{itemize}
	\item[(A)] In the interior-point method, we start with an estimate of the multiplier, and then, in every iteration, set $\mu_k = 10\|y^{(k)}\|_\infty$, possibly with a control of the size of the change; here the index $k$ refers to the interior-point iteration and $y$ is as in Lemma~\ref{th:lemma}.
	\item[(B)]
	Estimate the value of $\mu$, fully solve the penalized problem \cref{eq:BQPSDP2pen2} by Loraine, check the residuum of the equality constraint in \cref{eq:BQPSDP2} for the optimal solution of \cref{eq:BQPSDP2pen2} and, if bigger than a prescribed threshold, increase $\mu$ and repeat this process. 
\end{itemize}
Option (A) is  attractive as it is supported by the theory. However, by changing the penalty parameter in every iteration, we also change the optimization problem. This is, unfortunately, reflected in the efficiency of the IP method and of the preconditioner.
Option (B) is simple and robust but may also be expensive: we may need to solve several SDP problems to full optimality. The good news is that a class of BQP problems usually requires a single choice of $\mu$ in order to guarantee convergence to an optimal point with the norm of the residuum of the equality constraint smaller than a prescribed value; see the numerical experiments in the last section.
Needless to say that one cannot start with a too big value of $\mu$ due to the potential ill-conditioning of the resulting SDP problem and resulting difficulties of the IP algorithm to solve the problem at all.

\begin{remark}
	Instead of problem \cref{eq:BQPSDP2pen2}, we could use another reformulation of problem \cref{eq:BQPSDP2pen}, namely
	\begin{align*}
	&\min_{z\in \RR^{\tn}} (\svec I)^\top z + \mu \sum_{i=1}^n (\BM z - {q})_i\\
	&\mbox{subject to}~
	\begin{aligned}[t]& \smat(z) \succeq 0\\
	&\BM z - {q}\geq 0\,,
	\end{aligned}\nonumber
	\end{align*}
	where we use the simple fact that $\|x\|_1 = \sum x_i$ for $x\geq 0$.
	This problem is also in the primal form \cref{eq:SDP-P} and may seem more attractive than \cref{eq:BQPSDP2pen2}, as it does not need the additional variables~$s$. We have implemented and tested this formulation, too. It turned out that our solvers were more efficient when solving problem \cref{eq:BQPSDP2pen2}, both in terms of number of iterations (in particular in the ADMM method) and the efficiency of the preconditioner in Loraine. Hence all numerical experiments reported here use problem formulation~\cref{eq:BQPSDP2pen2}.
\end{remark}

\subsection{Loraine and problem \cref{eq:BQPSDP2pen2}}
Recall the definition of the matrix $\BM= (\svec M_1, \ldots,\svec M_n)^T$ with matrices $M_i$ introduced at the beginning of this section. We have the following result.
\begin{lemma}\label{th:chordal}
	There exists a permutation matrix $P\in\RR^{\tn\times\tn}$ such that $P\BM^\top\BM P^\top$ is a block diagonal matrix with small full blocks. In particular, $\BM^\top\BM$ is a  sparse chordal matrix.
\end{lemma}	
\begin{proof}
	Every matrix $M_i$ localizes the corresponding variables in the moment matrix $\sum\limits_{i=1}^n M_iy_i+I$ in \cref{eq:BQPSDPrel1}, whereas the elements of the moment matrix are either ones or the variables $y_i$; see also \Cref{ex:21}.
	As every $M_i$ is associated with exactly one variable $y_i$, the nonzeros in this matrix are unique to the matrix.
	Therefore every column of matrix $\BM$ will either contain all zeros or a single nonzero number one at the $i$-th position, $i=1,\ldots,n$. We can now re-order the columns of $\BM$ such that the first columns will only contain zeros, the next columns element 1 at the first position, the next ones element 1 at the second position, etc. This re-ordering will define the permutation matrix $P$. The claims follow.
\end{proof}
\begin{Example} \label{ex:22}
	Consider a problem with $s=3$ and relaxation order $\omega=2$. Utilizing the constraints $x_i^2=1$,  
	we only consider variables associated with the monomial basis $b(x)=(x_1,\ x_2,\ x_3,\ x_1x_2,\ x_2x_3,\ x_1x_3)$:
	{\small
		\begin{align*}
		X = b(x)b^\top(x)&=\begin{pmatrix}
		x_1^2&x_1x_2&x_1x_3&x_1^2x_2&x_1x_2x_3&x_1^2x_3\\
		x_1x_2&x_2^2&x_2x_3&x_1x_2^2&x_2^2x_3&x_1x_2x_3\\
		x_1x_3&x_2x_3&x_3^2&x_1x_2x_3&x_2x_3^2&x_1x_3^2\\
		x_1^2x_2&x_1x_2^2&x_1x_2x_3&x_1^2x_2^2&x_1x_2^2x_3&x_1^2x_2x_3\\
		x_1x_2x_3&x_2^2x_3&x_2x_3^2&x_1x_2^2x_3&x_2^2x_3^2&x_1x_2x_3^2\\
		x_1^2x_3&x_1x_2x_3&x_1x_3^2&x_1^2x_2x_3&x_1x_2x_3^2&x_1^2x_3^2
		\end{pmatrix}\\
		&=\begin{pmatrix}
		1&x_1x_2&x_1x_3&x_2&x_1x_2x_3&x_3\\
		x_1x_2&1&x_2x_3&x_1&x_3&x_1x_2x_3\\
		x_1x_3&x_2x_3&1&x_1x_2x_3&x_2&x_1\\
		x_2&x_1&x_1x_2x_3&1&x_1x_3&x_2x_3\\
		x_1x_2x_3&x_3&x_2&x_1x_3&1&x_1x_2\\
		x_3&x_1x_2x_3&x_1&x_2x_3&x_1x_2&1
		\end{pmatrix}
		\end{align*}
	}%
	i.e., variables $y\in\RR^7$ corresponding to $(x_1, x_2, x_3, x_1x_2,\ x_2x_3, x_1x_3,x_1x_2x_3)$, the unique elements of~$X$.
	The matrices $M_i$ simplify to
	{\small
		\begin{align*}
		M_1 &= \begin{pmatrix}
		0&0&0&0&0&0\\0&0&0&1&0&0\\0&0&0&0&0&1\\0&1&0&0&0&0\\0&0&0&0&0&0\\0&0&1&0&0&0
		\end{pmatrix},\ 
		M_2 = \begin{pmatrix}
		0&0&0&1&0&0\\0&0&0&0&0&0\\0&0&0&0&1&0\\1&0&0&0&0&0\\0&0&1&0&0&0\\0&0&0&0&0&0
		\end{pmatrix},\ldots,
		M_7 = \begin{pmatrix}
		0&0&0&0&1&0\\0&0&0&0&0&1\\0&0&0&1&0&0\\0&0&1&0&0&0\\1&0&0&0&0&0\\0&1&0&0&0&0
		\end{pmatrix}
		\end{align*}
	}%
	and thus
	{\small
		$$
		\BM = 
		\left(\begin{array}{ccccccccccccccccccccc}
		0&0&0&0&0&0&0&1&0&0&0&0&0&0&0&0&0&1&0&0&0\\
		0&0&0&0&0&0&1&0&0&0&0&0&1&0&0&0&0&0&0&0&0\\
		0&0&0&0&0&0&0&0&0&0&0&1&0&0&0&1&0&0&0&0&0\\
		0&1&0&0&0&0&0&0&0&0&0&0&0&0&0&0&0&0&0&1&0\\
		0&0&0&0&1&0&0&0&0&0&0&0&0&0&0&0&0&0&1&0&0\\
		0&0&0&1&0&0&0&0&0&0&0&0&0&1&0&0&0&0&0&0&0\\
		0&0&0&0&0&0&0&0&1&0&1&0&0&0&0&0&1&0&0&0&0
		\end{array}\right)
		$$
	}%
	and
	{\small
		$$
		\BM P^\top = 
		\left(\begin{array}{ccccccccccccccccccccc}
		0&0&0&0&0&0&1&1&0&0&0&0&0&0&0&0&0&0&0&0&0\\
		0&0&0&0&0&0&0&0&1&1&0&0&0&0&0&0&0&0&0&0&0\\
		0&0&0&0&0&0&0&0&0&0&1&1&0&0&0&0&0&0&0&0&0\\
		0&0&0&0&0&0&0&0&0&0&0&0&1&1&0&0&0&0&0&0&0\\
		0&0&0&0&0&0&0&0&0&0&0&0&0&0&1&1&0&0&0&0&0\\
		0&0&0&0&0&0&0&0&0&0&0&0&0&0&0&0&1&1&0&0&0\\
		0&0&0&0&0&0&0&0&0&0&0&0&0&0&0&0&0&0&1&1&1
		\end{array}\right)
		$$
	}%
	with a suitable permutation matrix $P$. Then $P\BM^\top\BM P^\top$ is block diagonal with block sizes 2 and~3.
\end{Example}
\begin{remark}
	The size of the blocks in the matrix $P\BM^\top\BM P^\top$ is given by the count of the corresponding variables in the upper triangle of the moment matrix. This number grows with the relaxation order.
\end{remark}
\begin{theorem}
	Assume that \cref{eq:BQP} has a unique solution and that Lasserre hierarchy is exact for order $\omega \geq 1$. Then problem \cref{eq:BQPSDP2pen2} corresponding to $\omega$ satisfies Assumptions \ref{assum:a2}--\ref{assum:alr}.
\end{theorem}
\begin{proof}
	A strictly feasible point for \cref{eq:BQPSDP2pen2}  can be found, e.g., by choosing $z$ as a unit vector and $s$ positive, arbitrarily large. Similarly, the moment problem \cref{eq:BQPSDPrel1} is trivially strictly feasible by choosing $y=0$. Hence \Cref{assum:a2} holds.
	Every matrix ${\cal E}_i$ contains at most two nonzero elements, hence \Cref{assum:a3} is trivially satisfied.
	\Cref{assum:a4} holds by Lemma~\ref{th:chordal}: because the matrix $\BM^\top\BM$ is chordal and sparse, sparse Cholesky factorization leads to zero fill-in and is thus very efficient.
	\Cref{assum:a5} is satisfied by construction of the problem, as  the number of variables $\tn$ is proportional to the square of the size of the matrix inequality $m$.
	Finally, under the assumption of unique solution to \cref{eq:BQP} and exact relaxation, the dual variable associated with the matrix inequality in \cref{eq:BQPSDP2pen2} has rank at most two. Hence \Cref{assum:alr} is satisfied.
\end{proof}

\section{ADMM for SDP relaxations}
The Alternating Direction Method of Multipliers (ADMM) is a popular alternative to the interior-point methods for convex optimization problems. Its extension to semidefinite optimization was introduced in \cite{wen2010alternating}. It proved to be efficient for certain SDP problems, among others many UBQPs (\cite[Sec 4.3]{wen2010alternating}). We thus briefly describe this algorithm and offer a numerical comparison with Loraine. Furthermore, we will introduce a new ``hybrid" approach combining ADMM with Loraine.

\subsection{The ADMM-SDP algorithm}
The method consists of minimizing the augmented Lagrangian of the problem
$$
{\cal L}_\rho(y,S,X) = -b^\top y +  X \bullet \left(\sum_{i=1}^n y_i A_i + S - C\right) 
+ \frac{1}{2\rho} \left\|\sum_{i=1}^n y_i A_i + S - C\right\|_F^2
$$
alternately, first with respect to $y$, then with respect to $S$, and then updating $X$ by
$$
X^{(k+1)} = X^{(k)} + \frac{\sum_{i=1}^n y^{(k+1)}_i A_i + S^{(k+1)} - C}{\rho}\,.
$$
While the minimization with respect to $y$ leads to an unconstrained problem and the minimum is obtained by solving the first-order optimality system of (linear) equations, the
minimization of ${\cal L}_\rho$ with respect to $S$ can be formulated as a projection of the matrix $C-\sum_{i=1}^n y_i A_i - \rho X$ onto the semidefinite cone $\S^m_+$.

For $Z\in\S^m$, let ${\cal A}(Z)$ be a vector with elements $A_i\bullet Z$, $i=1,\ldots,n$ and $\left[Z\right]_+$ the projection of $Z$ on the semidefinite cone.
With a penalty parameter $\rho>0$ and a relaxation parameter $\sigma\in(0,\frac{1+\sqrt{5}}{2})$, the update of the primal-dual point $(y,S,X)$ in the $k$-th iteration of ADMM is defined in \Cref{algo:admm}.
\begin{algorithm}
	\caption{ADMM for SDP}
	\label{algo:admm}
	Given an initial point $(y^{(0)},S^{(0)},X^{(0)})$, parameters $\rho$ and $\sigma$, and a stopping parameter $\varepsilon_{\scriptscriptstyle\textrm{ADMM}}$.
	\begin{algorithmic}[1]
		\For{$k=0,1,2,\ldots$}
		\State Update $y$ by
		$y^{(k+1)} = -(\BA\BAT)^{-1}\left(\rho({\cal A}(X^{(k)})-b) + {\cal A}(S^{(k)}-C)\right)$
		\State Update $S$ by
		$\widehat{S} = C - \sum_{i=1}^n y_i^{(k+1)} A_i - \rho X^{(k)}, \quad S^{(k+1)} = \left[\widehat{S}\right]_+$
		\State Update $X$ by
		$\widehat{X} = \frac{1}{\rho}\left(S^{(k+1)} - \widehat{S}\right),\quad
		X^{(k+1)} = (1-\sigma)X^{(k)} + \sigma \widehat{X}$
		\State Update $\rho$
		\State Check convergence
		\EndFor
	\end{algorithmic}
\end{algorithm}

For details, in particular the choice and update of $\rho$, see \cite{wen2010alternating}.
In the numerical experiments below, we use our MATLAB implementation\footnote{https://github.com/kocvara/ADMM\_for\_SDP} of the algorithm.

\subsection{ADMM as a warm starter for Loraine}
When testing the ADMM algorithm with the MAXCUT problems considered in this paper, we have made the following two observations:
\begin{itemize}
	\item[(i)] The method is more efficient when solving the reformulated (bigger) problem \cref{eq:BQPSDP2pen2} than the smaller one \cref{eq:BQPSDPrel1}. Moreover, for problem \cref{eq:BQPSDP2pen2}, the convergence is very fast in the first iterations, before it slows down: the algorithm gets quickly a ``good" approximation of the solution, both primal and dual.
	\item[(ii)] Due to the nature of the ADMM method (projections on the semidefinite cone), the rank of the approximate solution is always ``numerically exact''; in particular, the rank of $X^{(k)}$ is low when we are close enough to the solution.
\end{itemize}
The fact that a low-rank approximation of the solution can be obtained relatively quickly, together with the fact that ADMM is a primal-dual algorithm leads to the following idea: use a low-precision ADMM solution as a warm start for Loraine. Presuming that the ADMM approximation of $X$ is already of the expected low rank (or very close to it), the preconditioner $H_\alpha$ will be extremely efficient during the remaining iterations of Loraine. We therefore propose the following \emph{hybrid ADMM-Loraine} \Cref{algo:admm-loraine}.
\begin{algorithm}
	\caption{ADMM-Loraine}
	\label{algo:admm-loraine}
	\begin{algorithmic}[1]
		\State
		Solve \cref{eq:BQPSDP2pen2} by ADMM with stopping tolerance $\varepsilon_{\scriptscriptstyle\textrm{ADMM}}$. Save the primal-dual approximate solution $(y,S,X)_{\scriptscriptstyle\textrm{ADMM}}$.
		\State
		Solve \cref{eq:BQPSDP2pen2} by Loraine with
		\begin{itemize}
			\item[--] initial point $(y,S,X)_{\scriptscriptstyle\textrm{ADMM}}$
			\item[--] initial value of the stopping criterion for the CG method reduced to $10^{-6}$.
		\end{itemize}
	\end{algorithmic}
\end{algorithm}

Notice that, in order to get a primal-dual approximation of the solution, we have to solve the same problem formulation by ADMM and Loraine, i.e., formulation \cref{eq:BQPSDP2pen2} with the $\ell_1$-penalty and the same value of the penalty parameter $\mu$, despite the fact that ADMM could directly handle the linear equality constraints.

\section{Numerical experiments}\label{sec:examples}

In this section we report the results of our numerical experiments with relaxations of MAXCUT problems and randomly generated UBQPs. The presented algorithms will be compared with the state-of-the-art SDP solver MOSEK \cite{mosek} that is relying on the interior-point method and uses a direct linear solver for the Schur complement equation.

Recall from the introduction that \emph{it is not our goal to compete with current most efficient UBQP solvers}, rather to offer, hopefully, ways to increase their efficiency by using the higher order relaxations. Therefore we do not include results for any branch-and-bound-type solver, as our goal is merely to solve the relaxation \mbox{(sub-)problems}.

The stopping criterion for our implementations of Loraine and ADMM was based on the DIMACS errors \cite{dimacs} that measure the (normalized) primal and dual feasibility, duality gap and complementary slackness. The algorithms were stopped when all these errors were below $10^{-6}$. MOSEK was used with default settings.

All problems were solved on an iMac desktop computer with 3.6 GHz 8-Core Intel Core~i9 and 64 GB 2667 MHz DDR4 using MATLAB R2022b.
\subsection{Problem dimensions}
In the following experiments, we will compare algorithms solving two different formulation of the problem,  the original one~\cref{eq:BQPSDPrel1} and the one  with $\ell_1$-penalty \cref{eq:BQPSDP2pen2}. It is therefore worth to first compare the dimensions of these formulations.

First recall from \cref{eq:dim} the relation of the size $s$ of the original UBQP problem with the dimensions $n,m$ of the SDP relaxation \cref{eq:BQPSDPrel1}:
$
\frac{s^4}{24} \lessapprox n \leq 2^s - 1 ,\ m = \sum_{i=1}^\omega {s\choose i}\,.
$
\Cref{tab:problem_sizes} gives a comparison of problem dimensions  in \cref{eq:BQPSDPrel1} and \cref{eq:BQPSDP2pen2} in terms of $s$ and~$n$.
\begin{table}[htbp]
	\centering
	\caption{Problem dimensions for formulations \cref{eq:BQPSDPrel1} and \cref{eq:BQPSDP2pen2} as functions of the relaxation order; here $s$ is the number of variables in BQP \cref{eq:BQP} and $\omega$ the relaxation order and $\displaystyle\frac{s^4}{24} \lessapprox n \leq 2^s - 1$}
	\label{tab:sisi}\label{tab:problem_sizes}	
	{\footnotesize
		\begin{tabular}{cc|ccc}\toprule
			\multicolumn{2}{c|}{problem \cref{eq:BQPSDPrel1}}&		\multicolumn{3}{c}{problem \cref{eq:BQPSDP2pen2}}		\\
			variables &	matrix size &	variables &	matrix size &	lin. constraints\\\midrule
			$n$& $\displaystyle\sum_{i=1}^\omega {s\choose i}$ &	$\displaystyle\left(\sum_{i=1}^\omega {s\choose i}\right)^{\!\!2}+n$ &	$\displaystyle\sum_{i=1}^\omega {s\choose i}$ & $2n$\\
			\bottomrule
	\end{tabular}}
\end{table}
\subsection{MAXCUT problems}
For the first set of numerical experiments, we use MAXCUT problems, the standard test problems for UBQP algorithms.
Let $\Gamma$ be an undirected $n$-node graph and
let the arcs $(i,j)$ be associated with nonnegative weights $a_{ij}$.
The MAXCUT problem is formulated as the following UBQP:
\begin{equation}\label{eq:maxcut1}
\max_x\left\{\frac{1}{4}\sum_{i,j=1}^n a_{ij}(1-x_ix_j)
\mid x_i^2=1,\ i=1,\ldots,n\right\}
\end{equation}
and can be thus solved by the techniques presented above. The solution of the SDP relaxation \cref{eq:SDP} is of rank two whenever the relaxation is exact and the solution to \cref{eq:maxcut1} is unique. (Recall that uniqueness of $x$ means uniqueness up to the multiple by $-1$.)

Obviously, a complete unweighted graph (with $a_{ij}\in\{0,1\}$) may have many ``symmetric" solutions. Therefore we generated two sets of problems of increasing size. Firstly, to avoid the non-uniqueness, we generated undirected, weighted, generally complete graphs with weights randomly distributed between 0 and 12 with 20--50 nodes, using the MATLAB command {\tt graph}. It turned out that for all these problems the relaxation order two in the Lasserre hierarchy is already high enough to deliver the optimal solution of the MAXCUT problem \cref{eq:maxcut1}. Secondly, we generated a set of problems of the same sizes but for unweighted graphs.
The dimensions of the corresponding SDP problems for relaxation order two are shown in Table~\ref{tab:maxcut_problems}.
\begin{table}[tbhp]
	\centering
	\caption{Problems \texttt{MAXCUT-<n>}, relaxation order $\omega=2$: dimensions for formulations \cref{eq:BQPSDPrel1} and \cref{eq:BQPSDP2pen2}}
	\label{tab:maxcut_problems}	
	{\footnotesize
		\begin{tabular}{c|rr|rrr}\toprule
			&\multicolumn{2}{c|}{problem \cref{eq:BQPSDPrel1}}&		\multicolumn{3}{c}{problem \cref{eq:BQPSDP2pen2}}		\\
			problem & variables & matrix size & variables & matrix size & lin.\ constr.\\\midrule
			MAXCUT-20	&   6\,196&    211&     22\,366 &    211 &   6\,196\\
			MAXCUT-25	&  15\,275&    326&     53\,301 &    326 &  15\,275\\
			MAXCUT-30	&  31\,930&    466&    108\,811 &    466 &  31\,930\\
			MAXCUT-35	&  59\,535&    631&    199\,396 &    631 &  59\,535\\
			MAXCUT-40	& 102\,090&    821&    337\,431 &    821 & 102\,090\\
			MAXCUT-45	& 164\,220& 1\,036&    537\,166 & 1\,036 & 164\,220\\
			MAXCUT-50	& 251\,175& 1\,276& 814\,726 & 1\,276 & 251\,175\\
			\bottomrule
	\end{tabular}}
\end{table}

We first present the results for the weighted problems, see \Cref{tab:MAXCUT}. These problems are relatively ``simple": the ADMM method is rather efficient, and MOSEK converges in a very small number of iterations, six to eight. However, the sheer size prevents MOSEK with a direct solver to solve larger problems. (Notice that to solve formulation \cref{eq:BQPSDPrel1} of problem MAXCUT-40 by an interior-point method, one has to build and solve a sparse linear system with a 102\,090$\times $102\,090 matrix. This is rather challenging for a direct solver.) We present results for ADMM applied to both problem formulations, the original one~\cref{eq:BQPSDPrel1} and the re-written one of larger dimension and with $\ell_1$-penalty \cref{eq:BQPSDP2pen2}. Perhaps surprisingly, ADMM is more efficient for the latter problem, despite the bigger dimension. Moreover, unlike for ADMM applied to \cref{eq:BQPSDPrel1}, the number of iterations of ADMM applied to \cref{eq:BQPSDP2pen2} is almost independent of the problem size.
The last four columns of \Cref{tab:MAXCUT} show the results for the hybrid ADMM-Loraine algorithm: we present the number of iterations of ADMM plus Loraine (`iter'), the time of ADMM (`timeA'), time of Loraine (`timeL') and the total time of the hybrid method. For these experiments, we have used the stopping tolerance of ADMM
$\varepsilon_{\scriptscriptstyle\textrm{ADMM}}=5\cdot 10^{-3}$ (problems 20--25) and $\varepsilon_{\scriptscriptstyle\textrm{ADMM}}=5\cdot 10^{-4}$ (problems 30--50).
\begin{table}[htbp]
	\centering
	\caption{Loraine, MOSEK, ADMM and ADMM-Loraine in weighted \texttt{MAXCUT-<n>} problems, relaxation order $\omega=2$}
	\label{tab:MAXCUT}%
	{\footnotesize
		\begin{tabular}{c|rrr|r|rr|rr|rrrr}
			\toprule
			\multicolumn{1}{c|}{{\sc maxcut}} & \multicolumn{3}{c|}{Loraine for \cref{eq:BQPSDP2pen2}} & \multicolumn{1}{c|}{{\sc mosek} \cref{eq:BQPSDPrel1}} & \multicolumn{2}{c|}{{\sc admm} for \cref{eq:BQPSDPrel1}} & \multicolumn{2}{c|}{{\sc admm} for \cref{eq:BQPSDP2pen2}}&\multicolumn{4}{c}{{\sc admm}-Loraine for \cref{eq:BQPSDP2pen2}}\\
			\multicolumn{1}{c|}{problem} & \multicolumn{1}{|r}{iter} & \multicolumn{1}{r}{CG it} & \multicolumn{1}{r|}{time} & \multicolumn{1}{r|}{time}& \multicolumn{1}{r}   {iter} & \multicolumn{1}{r|}{time} & \multicolumn{1}{r}   {iter} & \multicolumn{1}{r|}{time} 
			& \multicolumn{1}{r}{iter} & \multicolumn{1}{r}{timeA} & \multicolumn{1}{r}{timeL} & \multicolumn{1}{r}{time}
			\\\midrule
			20 & 18 &  559 & 3.8 &    9 &  3615 &   10 & 3042 &  13 & 628+4 & 3.9& 0.8& 4.7\\
			25 & 20 &  728 &  11 &   78 &  4732 &   33 & 2735 &  28 & 181+5 & 2.5& 2.9& 5.4\\
			30 & 21 & 1032 &  28 &  607 &  6770 &   99 & 3537 &  80 & 795+5 & 20 & 5.8&  26\\
			35 & 23 & 2183 &  96 & 2911 &  5255 &  164 & 3030 & 126 & 863+4 & 42 &  11&  53\\
			40 & 27 & 2275 & 186 &  mem &  9611 &  500 & 1280 &  92 & 914+7 & 73 & 132& 205\\
			45 & 25 & 2521 & 335 &      & 16901 & 1400 & 2639 & 358 & 755+4 & 104&  33& 137\\
			50 & 24 & 2540 & 528 &      & 19521 & 2951 & 3296 & 745 & 727+5 & 162&  58& 220\\
			\bottomrule
	\end{tabular}}%
\end{table}%

To better demonstrate the behaviour of the hybrid method, in \Cref{fig:MAXCUT_hyb} we show the output of both codes. We can see that the ADMM method, indeed, gets very quickly a good approximation of objective value, while Loraine uses the warm start very efficiently. 
\begin{figure}
	{\footnotesize
		\begin{verbatim}
		================================= A D M M  for  S D P ==================
		Number of LMI Constraints:   1
		Number of Variables: 1065901
		Maximal Constraint Size: 1276
		------------------------------------------------------------------------
		iter    p-infeas     d-infeas       d-gap        error      objective
		------------------------------------------------------------------------
		100   0.00058292   0.08603023   0.00093126    0.086030   1981.22444814
		200   0.00018776   0.02554937   0.00034476    0.025549   2028.80448106
		300   0.00018959   0.01054250   0.00063882    0.010542   2037.35317708
		400   0.00010212   0.00991976   0.00025128    0.009920   2025.69453823
		500   0.00005383   0.00972192   0.00013810    0.009722   2038.21968191
		600   0.00002508   0.00891292   0.00001945    0.008913   2026.90960413
		700   0.00002609   0.00281354   0.00010411    0.002814   2037.33886004
		727   0.00002899   0.00047182   0.00012837    0.000472   2039.90970508
		------------------------------------------------------------------------
		Total ADMM iterations: 727; Final precision: 4.72e-04; CPU Time 162.59s
		========================================================================
		
		*** Loraine v0.1 ***
		Number of variables: 1065901
		Matrix size(s)     :    1276
		Linear constraints :  502350
		*** IP STARTS
		it      objective      error   cg_iter   CPU/it
		1   2.03994277e+03  3.77e-03       36     9.39
		2   2.03999424e+03  2.32e-03       33     9.55
		3   2.03999941e+03  2.50e-04       20     7.71
		4   2.03999992e+03  3.84e-05        9     5.80
		5   2.03999999e+03  5.96e-06        8     5.54
		*** Total CG iterations:    106
		*** Total CPU time:         58.24 seconds
		\end{verbatim}
	}
	\caption{Printout of low-precision ADMM and warm-started Loraine for the weighted \texttt{MAXCUT-50}.}
	\label{fig:MAXCUT_hyb}
\end{figure}

Next we will try to solve the unweighted MAXCUT problems. Recall that these problems often have nonunique solutions, in particular problems with almost dense graphs.
Moreover, relaxation order $\omega=2$ may not be high enough to obtain an exact solution.
We thus cannot expect the preconditioner to be efficient in those cases. This is, indeed, demonstrated in \Cref{tab:MAXCUT_un}. This table, in addition to columns identical to \Cref{tab:MAXCUT}, shows also the rank of the solution of the order-2 relaxation. As expected, Loraine is less efficient for problems with higher solution rank, in particular problems \texttt{MAXCUT-30} and \texttt{MAXCUT-50}. Still, it can solve these problems reliably. Again, the hybrid method is superior for these problems.
The convergence behaviour of the algorithms and their estimated complexity are further illustrated in \Cref{fig:1a}.
\begin{table}[htbp]
	\centering
	\caption{Loraine, MOSEK, ADMM and ADMM-Loraine in unweighted \texttt{MAXCUT-<n>} problems, relaxation order $\omega=2$}
	\label{tab:MAXCUT_un}%
	{\footnotesize
		\begin{tabular}{c|rrr|r|rr|rrrr|r}
			\toprule
			\multicolumn{1}{c|}{{\sc maxcut}} & \multicolumn{3}{c|}{Loraine for \cref{eq:BQPSDP2pen2}} & \multicolumn{1}{c|}{{\sc mosek} \cref{eq:BQPSDPrel1}} & \multicolumn{2}{c|}{{\sc admm} for \cref{eq:BQPSDP2pen2}}&\multicolumn{4}{c|}{{\sc admm}-Loraine for \cref{eq:BQPSDP2pen2}}&\\
			\multicolumn{1}{c|}{problem} & \multicolumn{1}{|r}{iter} & \multicolumn{1}{r}{CG it} & \multicolumn{1}{r|}{time}  & \multicolumn{1}{r|}{time} & \multicolumn{1}{r}   {iter} & \multicolumn{1}{r|}{time} 
			& \multicolumn{1}{r}{iter} & \multicolumn{1}{r}{timeA} & \multicolumn{1}{r}{timeL} & \multicolumn{1}{r|}{time}&\multicolumn{1}{r}{rank}
			\\\midrule
			20 & 17 &  735 &  4.3 &    9 & 1981 &    8 &  738+7 & 3.4 & 1.5 & 4.9 &  2 \\
			25 & 18 &  858 &   14 &   82 &  805 &    9 &  328+4 & 3.5 & 2.8 & 6.3 &  4 \\
			30 & 23 & 6482 &  145 &  635 & 2218 &   53 &  566+8 &  13 & 135 & 148 & 57 \\
			35 & 21 & 2604 &  128 & 3357 & 5251 &  229 & 1222+8 &  48 &  15 &  63 &  2 \\
			40 & 22 & 3538 &  305 &  mem & 7137 &  575 & 1855+7 & 142 &  35 & 177 &  2 \\
			45 & 22 & 3225 &  536 &      & 6794 &  845 & 1081+4 & 139 &  49 & 188 &  4 \\
			50 & 25 &12599 & 2593 &      & 5723 & 1162 & 1677+5 & 361 & 365 & 726 & 13 \\
			\bottomrule 
	\end{tabular}}%
\end{table}%
%
\begin{figure}[htbp]
	\begin{center}
		\resizebox{0.45\hsize}{!}
		{\includegraphics{{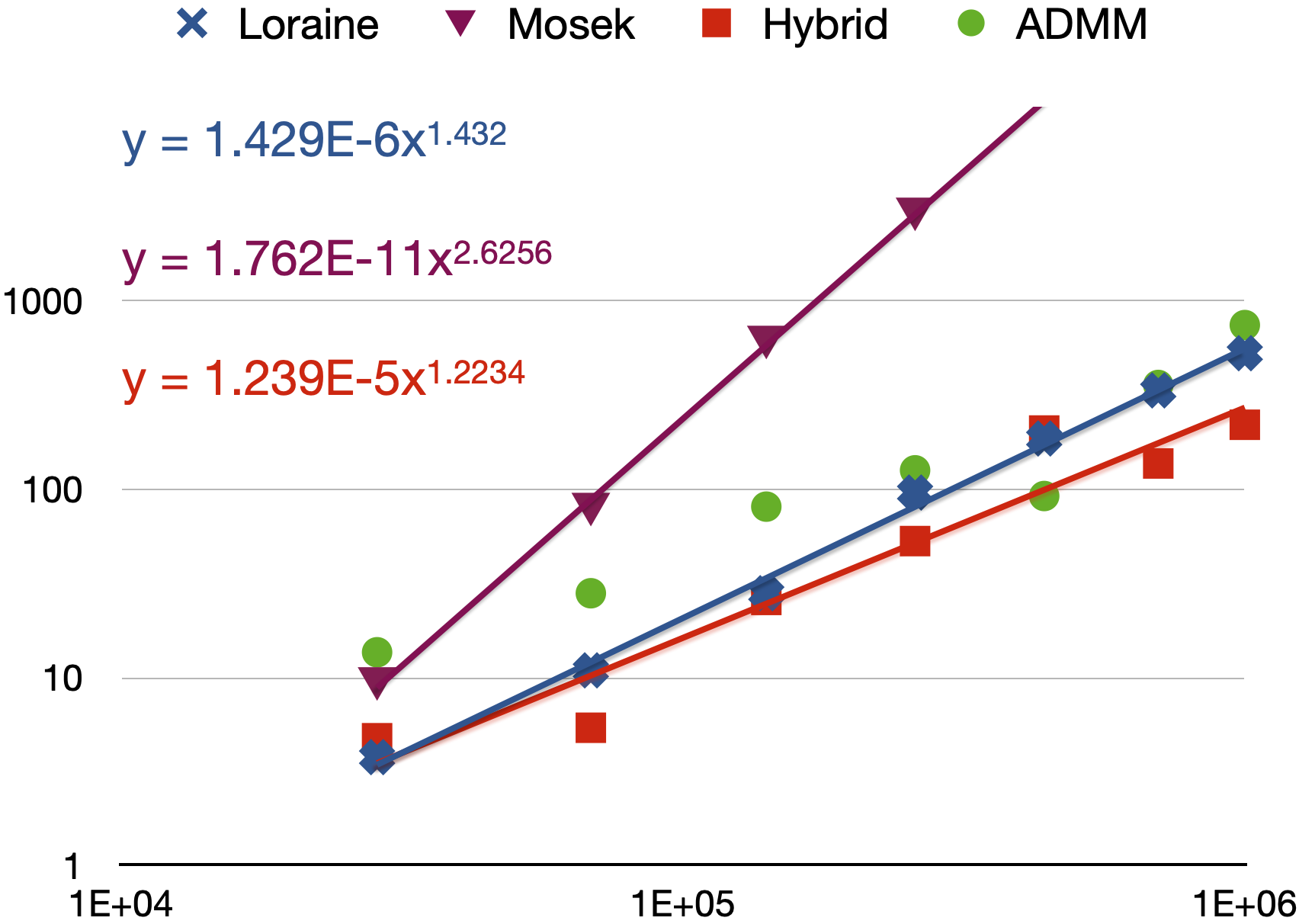}}}
		\resizebox{0.45\hsize}{!}
		{\includegraphics{{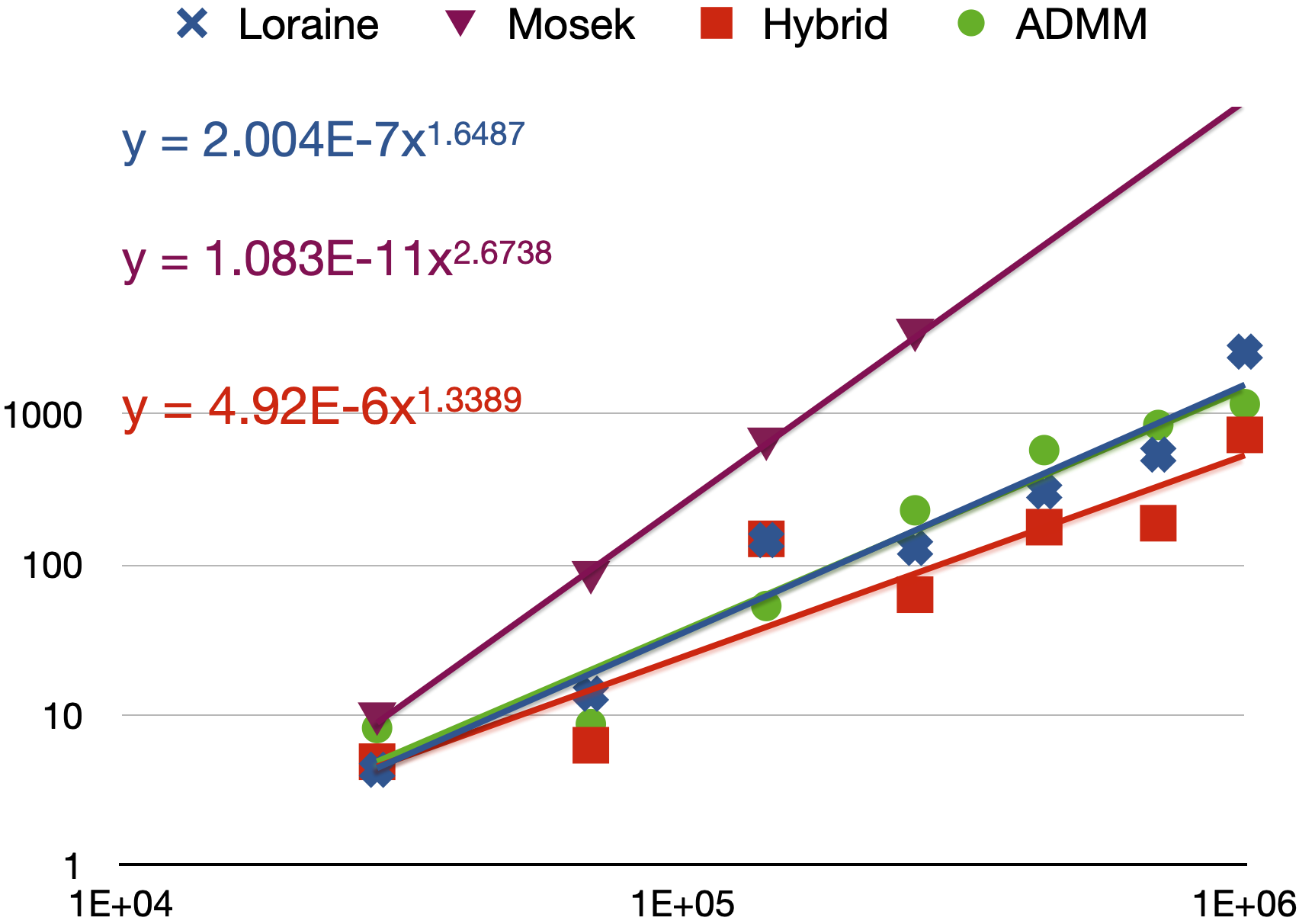}}}
	\end{center}
	\caption{CPU times in log-log scale for weighted (left) and unweighted (right) MAXCUT problems. Loraine (blue), MOSEK (brown), ADMM (green) and hybrid ADMM-Loraine (red).}
	\label{fig:1a}
\end{figure}

\subsection{Randomly generated problems}
We further conducted numerical experiments with matrices $Q$ randomly generated. We are aware of the fact that problems with random data may not always be representative and may sometimes lead to false conclusions regarding algorithm behaviour. However, we believe that these results still well demonstrate the efficiency of our approach.

It has been observed in \cite{kim-kojima} that problems with rank-one matrix $Q$ may require relaxation order of up to $\omega=\lceil s/2\rceil$ to reach the exact solution; a typical example is $Q=ee^{\!\top}$. On the other hand, for problems with matrix $Q$ of rank 3, $\omega=2$ was always sufficient in experiments performed in (\cite[Fig.4]{kim-kojima}). We have thus considered two main classes of problems:
\begin{itemize}
	\item[(A)] problems with ${\cal B}=\{-1,1\}$ and with $\rank Q=1$ generated by the following \mbox{MATLAB} code
	{\scriptsize
		\begin{verbatim}
		rng(0); q = randn(s,1); Q = q*q';
		\end{verbatim}}
	\item[(B)] problems with ${\cal B}=\{-1,1\}$ and with a full-rank indefinite $Q$ generated by the following MATLAB code
	{\scriptsize
		\begin{verbatim}
		rng(0); q = randn(s,1); Q = q*q';
		for k=1:s-1
		if ceil(k/2)*2 == k
		q = randn(s,1); Q = Q - q*q';
		else
		q = randn(s,1); Q = Q + q*q';
		end
		end
		\end{verbatim}}
\end{itemize}
Apart from random numbers generated with normal distribution, we also performed tests with uniform distribution (function \verb|rand|) and lognormal distribution (function \verb|logncdf|); in both cases the results and conclusions were rather similar to the above choice and are thus not reported here.

The exactness of the relaxation was measured by the numerical rank of the dual solution to the matrix inequality in problem \cref{eq:BQPSDP2pen2}---when the rank was equal to 1 or~2 (depending on the set ${\cal B}$) the relaxation order was considered sufficient; 
see \Cref{def:ex}.

\begin{remark}
	For $Q$ constructed as in (B), relaxation order $\omega=2$ was sufficient to get an exact solution of \cref{eq:BQP}. This observation, though, cannot be extended to any full-rank matrix $Q$. For instance, for $Q=ee^{\!\top}+\Diag(d), d\in\RR^s$, the lowest relaxation order will be $\omega=\lceil s/2\rceil$. This is because $x_i^2=1$ and thus the diagonal elements of $Q$ will be irrelevant in the optimization process and the resulting problem will be equivalent to that with a rank-one matrix $Q=ee^{\!\top}$. 
\end{remark}

\subsubsection{Full-rank $Q$ and relaxation order $\omega = 2$}
Using the MATLAB code from point (B) above, we generated problems of growing dimension $s=10\ldots50$ and solved the corresponding SDP relaxations for order $\omega=2$. 
The dimensions of the generated problems are reported in \Cref{tab:bqp_problems}; the table shows problem sizes for the original SDP relaxation \cref{eq:BQPSDPrel1} and for the re-written problem \cref{eq:BQPSDP2pen2}.
\begin{table}[htbp]
	\centering
	\caption{Randomly generated UBQP problems, relaxation order $\omega=2$: dimensions for formulations \cref{eq:BQPSDPrel1} and \cref{eq:BQPSDP2pen2}}
	\label{tab:bqp_problems}
	{\footnotesize
		\begin{tabular}{c|rr|rrr}\toprule
			&\multicolumn{2}{c|}{problem \cref{eq:BQPSDPrel1}}&		\multicolumn{3}{c}{problem \cref{eq:BQPSDP2pen2}}		\\
			UBQP size &	variables &	matrix size &	variables &	matrix size &	lin. constraints\\\midrule
			10 &	385 &	56 &	1596 &	56 &	770\\
			15 &	1\,940   &	   121 &	  7\,381 &	   121 &	3\,880\\
			20 &	6\,195   &	   211 &	 22\,366 &	   211 &	12\,390\\
			25 &	15\,275  &	   326 &	 53\,301 &	   326 &	30\,550\\
			30 &	31\,930  &	   466 &	108\,811 &	   466 &	63\,860\\
			35 &	59\,535  &	   631 &	199\,396 &	   631 &	119\,070\\
			40 &	102\,090 &	   821 &	337\,431 &	   821 &	204\,180\\
			45 &	164\,220 &	1\,036 &	537\,166 &	1\,036 &	328\,440\\
			50 &	251\,175 &	1\,276 &	814\,726 &	1\,276 &	502\,350\\
			\bottomrule
	\end{tabular}}
\end{table}
The computational results are presented in \Cref{tab:bqp_res1} and clearly demonstrate the efficiency of Loraine. The computational complexity of Loraine (applied to \cref{eq:BQPSDP2pen2}) and MOSEK and ADMM (applied to \cref{eq:BQPSDPrel1}) is further illustrated in \Cref{fig:1}. We do not report on the hybrid ADMM-Loraine algorithm. That is because ADMM {applied to problem} \cref{eq:BQPSDP2pen2} (as required by the hybrid method) appears to be much less efficient than for the MAXCUT problems and so the hybrid algorithm is not even competitive to ADMM applied to~\cref{eq:BQPSDPrel1}.
\begin{table}[htbp]
	\centering
	\caption{Randomly generated UBQP problems, relaxation order $\omega=2$: Loraine, MOSEK and ADMM}
	\label{tab:bqp_res1}
	{\footnotesize
		\begin{tabular}{c|rrr|rr|rr}\toprule
			&		\multicolumn{3}{c|}{Loraine for \cref{eq:BQPSDP2pen2}}& \multicolumn{2}{c|}{{\sc mosek} for \cref{eq:BQPSDPrel1}}	&\multicolumn{2}{c}{{\sc admm} for \cref{eq:BQPSDPrel1}}\\
			UBQP   size &  iter &	CG iter &	time & iter & time&	iter & time \\\midrule
			10 &	 17 &	178 &	0.2 & 6 & 0.2&  1006&  0.4\\
			15 &	 17 &	291 &	1.1 & 5 & 1.1&  1966&  3.0\\
			20 &	 18 &	447 &	4.0 & 6 & 9.3&  3063&  12\\
			25 &	 18 &	501 &	10 & 7 & 81 &  5843&   63\\
			30 &	 19 &	577 &	 22 & 7 & 496&  8018&  174\\
			35 &	 21 &	650 &	 46 &  \multicolumn{2}{c}{memory}& 11493&  511\\
			40 &	 23 &	1057 &	121 &   &    & 17532&  1318\\
			45 &	 22 &	2368 &	334 &   &    & 45030& 4258\\
			50 &	 21 &  1809 &	410 &   &    & 21407& 3316\\
			\bottomrule
		\end{tabular}
	}
\end{table}
%
\begin{figure}[htbp]
	\begin{center}
		\resizebox{0.45\hsize}{!}
		{\includegraphics{{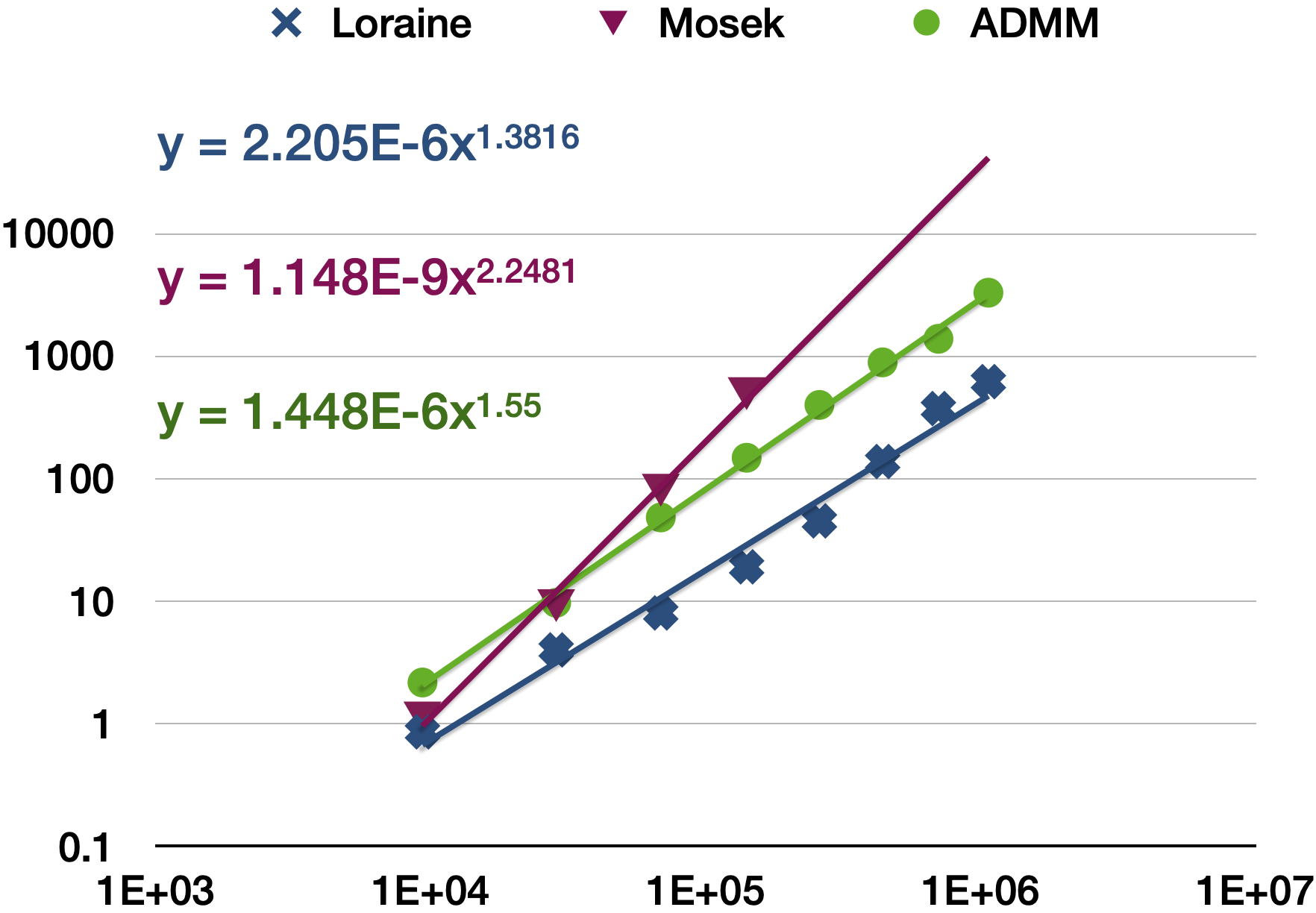}}}
	\end{center}
	\caption{CPU times in log-log scale for random BQP problems. Loraine (blue), MOSEK (brown) and ADMM (green).}
	\label{fig:1}
\end{figure}

\subsubsection{Rank-one $Q$ and higher relaxation order}
We know that, for these problems, relaxation order $\omega=2$ is typically not high enough to reach exact solution; order up to $\omega=\lceil s/2\rceil$ will be required.
Let us recall from \Cref{tab:sisi} the dependence of problem sizes on the relaxation order.
%
This table, in particular, shows that the number of variables grows much more quickly for problem \cref{eq:BQPSDP2pen2} than for the original problem \cref{eq:BQPSDPrel1} where it, eventually, reaches the finite limit $2^s-1$. Therefore we cannot expect Loraine with iterative solver applied to \cref{eq:BQPSDP2pen2} to be as efficient for higher-order relaxations as it is for \mbox{$\omega=2$}. This is clearly demonstrated in \Cref{tab:bqp_res2} comparing Loraine with a direct solver applied to problem \cref{eq:BQPSDPrel1} (this code was slightly faster than MOSEK for these problems, due to direct handling of the rank-one data matrices) with Loraine with the iterative solver applied to problem~\cref{eq:BQPSDP2pen2}. 
\begin{table}[htbp]
	\centering
	\caption{Randomly generated UBQP problem, $n=9$, relaxation order $\omega=2,\ldots,5$: problem dimensions and CPU times for Loraine-direct and Loraine-iterative}
	\label{tab:bqp_res2}
	{\footnotesize
		\begin{tabular}{c|rrr|rrrrr}\toprule
			&\multicolumn{3}{c|}{Loraine-direct for \cref{eq:BQPSDPrel1}}&		\multicolumn{5}{c}{Loraine-iterative for \cref{eq:BQPSDP2pen2}}		\\
			$\omega$ &	vars &	matrix size &	time  &	vars &	matrix size&	lin constr &	CG iter  &	time\\\midrule
			2  &	255 &	46 &	0.09&	1081 &	46 &	510 &	99 &	0.08\\
			3  &	465 &	130 &	0.77&	8515 &	130 &	930 &	203 &	0.97\\
			4  &	510 &	256 &	2.87&	32896 &	256 &	1020 &	302 &	15.7\\
			5  &	511 &	382 &	7.04&	73153 &	382 &	1022 &	701 &	129 \\
			\bottomrule
		\end{tabular}
	}
\end{table}



\section{Conclusions}
Our numerical experiments demonstrate the ability of an interior point method with a specially designed solver of the linear system to solve higher-order Lasserre relaxations of UBQP. The approach is particularly efficient for relaxation order two which, for most of the tested problems, was high enough to deliver either the exact solution of the UBQP or a good approximation of it. 

We have also introduced a new, hybrid algorithm that uses an approximate solution obtained by ADMM as a warm start for the used interior point method. This algorithm is rather efficient for problems for which ADMM itself is relatively efficient, such as the MAXCUT problems. On the other hand, it may be  inefficient once ADMM fails to converge quickly to an approximate solution.

\backmatter

\bmhead{Acknowledgments}
The work of the second author was co-funded by the European Union under the project ROBOPROX (reg.\ no.\  CZ.02.01.01/00/22\_008/0004590).

\section*{Declarations}

\noindent{\bf Conflict of interest.} The author has no competing interests to declare that are relevant to the content of this article.

\noindent{\bf Data availability. } The data used in Section \ref{sec:examples} were randomly generated and are available from the authors upon request.

\bibliographystyle{abbrv}
\bibliography{bibliography}

\end{document}